\documentclass[11pt,reqno]{amsart}

  \makeatletter
\@namedef{subjclassname@2020}{%
  \textup{2020} Mathematics Subject Classification}
\makeatother

\usepackage{graphicx} 
\usepackage{amsmath, amsfonts, amssymb, amscd,graphics, graphpap, tikz, color, xcolor,enumerate,enumitem, hyperref, mathabx, mathrsfs,float,cancel} 
 \usepackage[foot]{amsaddr}
\usepackage[normalem]{ulem}

\usetikzlibrary{arrows.meta,arrows}
\usetikzlibrary{decorations.pathreplacing,decorations.markings}
\tikzset{arrow data/.style 2 args={%
      decoration={%
         markings,
         mark=at position #1 with \arrow{#2}},
         postaction=decorate}
      }%
      
\usepackage[mathscr]{euscript}
  \usepackage{soul}

\usepackage{graphicx}

\makeatletter
\newcommand{\doublewidetilde}[1]{{%
  \mathpalette\double@widetilde{#1}%
}}

\textwidth= 
148mm
\numberwithin{equation}{section}

\theoremstyle{plain}
\newtheorem{theo}{Theorem}[section]
\newtheorem{lem}[theo]{Lemma}
\newtheorem{prop}[theo]{Proposition}

\newtheorem{cor}[theo]{Corollary}

\theoremstyle{definition}
\newtheorem{rem}[theo]{Remark}

\newtheorem{definition}[theo]{Definition}

\newenvironment{pf}{\noindent{\it Proof.\,}}{\hfill $\square$\par \medskip}

\theoremstyle{plain}

\theoremstyle{definition}

\newcommand{\beq}{\begin{equation}}
\newcommand{\eeq}{\end{equation}}
\renewcommand{\a}{\alpha}

\renewcommand{\d}{\delta}

\newcommand{\f}{\varphi}
\newcommand{\g}{\gamma}
\newcommand{\h}{\eta}

\renewcommand{\l}{\lambda}

\newcommand{\s}{\sigma}
\renewcommand{\t}{\tau}

\newcommand{\bR}{\mathbb{R}}

\newcommand{\bN}{\mathbb{N}}

\newcommand{\bA}{\mathbb{A}}
\newcommand{\bB}{\mathbb{B}}

\newcommand{\bW}{\mathbb{W}}

\newcommand{\bT}{\mathbb{T}}

\newcommand{\cA}{\mathscr{A}}
\newcommand{\cB}{\mathcal B}
\newcommand{\cC}{\mathcal{C}}

\newcommand{\cF}{\mathscr{F}}

\newcommand{\cJ}{\mathscr{J}}
\newcommand{\cK}{\mathscr{K}}

\newcommand{\cM}{\mathscr{M}}
\newcommand{\cN}{\mathscr{N}}
\newcommand{\cO}{\mathscr{O}}

\newcommand{\cQ}{\mathscr{Q}}

\newcommand{\cU}{\mathscr{U}}
\newcommand{\cV}{\mathscr{V}}

\newcommand{\p}{\partial}

\renewcommand{\square}{\kern1pt\vbox
{\hrule height 0.6pt\hbox{\vrule width 0.6pt\hskip 3pt
\vbox{\vskip 6pt}\hskip 3pt\vrule width 0.6pt}\hrule height0.6pt}\kern1pt}

\DeclareMathOperator\Id{Id}

\DeclareMathOperator{\Lie}{Lie}

\newcommand{\wt}{\widetilde}

\newcommand{\wh}{\widehat}

\newcommand{\bt}{\begin{theo}\ \ }
\newcommand{\et}{\end{theo}}
\newcommand{\bp}{\begin{prop}\ \ }
\newcommand{\ep}{\end{prop}}
\newcommand{\bc}{\begin{cor}\ \ }
\newcommand{\ec}{\end{cor}}
\newcommand{\bl}{\begin{lem}\ \ }
\newcommand{\el}{\end{lem}}
\newcommand{\bd}{\begin{definition}}
\newcommand{\ed}{\end{definition}}
\newcommand{\be}{\begin{equation}}
\newcommand{\ee}{\end{equation}}

\def\<#1,#2>{\langle\,#1,\,#2\,\rangle}
\newcommand{\arr}{\begin{array}{rlll}}
\newcommand{\ea}{\end{array}}
\newcommand{\bea}{\begin{eqnarray}}
\newcommand{\eea}{\end{eqnarray}}
\newcommand{\bean}{\begin{eqnarray*}}
\newcommand{\eean}{\end{eqnarray*}}

\renewcommand{\=}{:=}

\newcommand{\ve}{\varepsilon}

  \newcommand{\vertiii}[1]{{\left\vert\kern-0.25ex\left\vert\kern-0.25ex\left\vert #1 
    \right\vert\kern-0.25ex\right\vert\kern-0.25ex\right\vert}}

\newcommand{\Attain}{\cM\text{\it -Att}}

\newcommand{\orb}{\operatorname{Orb}}

\newcommand{\Kalman}{\operatorname{Kalman}}

\def\sideremark#1{\ifvmode\leavevmode\fi\vadjust{
\vbox to0pt{\hbox to 0pt{\hskip\hsize\hskip1em
\vbox{\hsize3cm\tiny\raggedright\pretolerance10000
\noindent #1\hfill}\hss}\vbox to8pt{\vfil}\vss}}}

\title[Flows of vector fields and  the  Kalman Theorem]{Flows of vector fields and  the  Kalman Theorem}
 \author[Bagagiolo, Giannotti,  Spiro and Zoppello]{Fabio Bagagiolo$^1$ \quad Cristina Giannotti$^*$$^2$\quad  Andrea Spiro$^2$ \quad Marta Zoppello$^3$}
 \address[$^1$]{Dipartimento di Matematica
Universit\`a di Trento
Via Sommarive, 14, 
I-38123 Povo (Trento)
ITALY\ \ \textit{E-mail}: \textnormal{fabio.bagagiolo@unitn.it}}

\address[$^2$]{Scuola di Scienze e Tecnologie
Universit\`a di Camerino
Via Madonna delle Carceri
I-62032 Camerino (Macerata)
ITALY\ \ \textit{E-mail}: \textnormal{cristina.giannotti@unicam.it, 
andrea.spiro@unicam.it}
}
\address[$^3$]{Dipartimento di Scienze Matematiche 
``G. L. Lagrange'' (DISMA)
Politecnico di Torino
Corso Duca degli Abruzzi, 24, 
10129 Torino 
ITALY\ \ \textit{E-mail}: \textnormal{marta.zoppello@polito.it}
}
\address[$^*$]{Corresponding author}
\begin{document}

\begin{abstract} 
 We give two   proofs of the Kalman Theorem,   alternative to the most common  ones,  which  infer such a classical result  of  Control Theory   using just  
 very basic facts on  flows of vector fields.  These  proofs are apt  to   be generalised  in diverse directions  -- in fact one of them has been already generalised, yielding new criteria for local controllability  of   non-linear real analytic controlled systems.
\end{abstract}
 \maketitle
\section{Introduction}
Consider a controlled dynamical system, whose states are represented by the  points $q = (q^i)_{ 1 \leq i \leq n}$  in the affine space  $\cQ \= \bR^n$ 
and with  evolutions   given by the solutions  of a  system of linear  first order equations  of the form
\beq\label{Thesystem-1} \dot q  = A q  + B u\eeq
for some non zero constant matrices  $A = (A^i_j)$ and $B = (B^i_a)$ and in which the controls $u(t)$ are measurable functions taking values  in  a  bounded set  $\cK \subset \bR^m$. For a fixed  $\overline q \in \cQ$,  denote by  $\cC_{\overline q}$ the set of all points $q\in \cQ$ that are {\it controllable to $\overline q$}, i.e.  for  which  there is a  control $u(t)$ which determines  a solution  $q(t)$ to \eqref{Thesystem-1}  that starts from $q$ and reaches $\overline q$ in a finite time.  A local version of the  famous Kalman Theorem consists of a criterion for determining whether  the    set   $\cC_{\overline q = 0_{\bR^n}}$ of controllable points to the origin contains an open set.   It can be stated as follows.  \\[5pt]
{\bf Theorem }(Kalman).  {\it  If $\cK$ contains a neighbourhood of $0_{\bR^m}$,  the set  $\cC_{\overline q = 0_{\bR^n}}$  contains  a  neighbourhood of  $0_{\bR^n}$  if and only if the  space $V \subset \bR^n$, which is  spanned by the vectors 
  \begin{multline} \label{tre}B_1, \ldots, B_m, \ \ A{\cdot}B_1, \ldots, A{\cdot} B_m,\ \  (A{\cdot} A){\cdot}  B_1,  (A{\cdot} A){\cdot}  B_1 \ldots, \ (A {\cdot} A){\cdot} B_m, \ldots\\
\ldots, \ \ \left(A{\cdot} \ldots {\cdot}A\right){\cdot}  B_1,  \left(A {\cdot}\ldots {\cdot}A\right){\cdot}  B_2, \ldots, \left(A{\cdot} \ldots {\cdot} A\right) {\cdot} B_m,\ \ \ldots \ ,
 \end{multline} 
has the maximal dimension  $n$ (here,  we denote by  $B_a $  the columns of $B $).}\par
The proof can be found in many textbooks (as for instance \cite{AS,Co}), and  it essentially uses   results in
Linear Algebra and Convex Geometry.\par
The  maximal dimension condition of the Kalman Theorem is very much reminiscent of  a known corollary  of the  Chow-Rashevski\u\i\ Theorem on  the local controllability of
 systems  of the form $\dot q = \sum_{a = 0}^m X_a(q) u^a$.  In fact,   
the sufficiency parts  of the Kalman Theorem and of the mentioned corollary of the Chow-Rashevski\u\i\  Theorem, can be  derived from  a  very general criterion by Sussmann (\cite{Su1} or  \cite[Thm. 3.29 Ex.3.3]{Co}; on this regard, see  also   foundational  Sussmann's  and Stefan's  papers   \cite{Su0, Su, St1, St2}). 
Furthermore, a couple of results by Hermann \cite{He} and Nagano  \cite{Na} provide a common argument for proving  the necessity parts of both theorems (on this regard, see also \cite{AS,Co}).
 Due to all this,   {\it it is commonly understood that  the  Kalman Theorem is  equivalent to  the  above quoted  corollary of the Chow-Rashevski\u\i\ Theorem}.  Since  it is not hard to check  that      ``the Kalman Theorem yields such a corollary  of the  Chow-Rashevski\u\i\ Theorem''  (see  e.g. \S \ref{KalmantoChow} below for a short proof or  \cite{El, ST} for  proofs holding  in  much more general settings), on the basis of the above common belief  it is sensible to expect the existence of a  straightforward way of  deriving   the  Kalman Theorem  from the Chow-Rashevski\u\i\ Theorem.   In this paper we make   explicit such a direct implication in two ways. 
More precisely, our main achievements  consist of   two (at best of our knowledge,  new) proofs   of  the  Kalman Theorem,  which also appear to be  appropriate  for   generalisations to fully nonlinear control systems. Actually,  the second  of our proofs  have   been successfully exploited  
in \cite{GSZ} and have led to    new  controllability   criteria  for nonlinear real analytic control system.\par
 Our first proof is  short and gives a simple way  to pass from   the (above described corollary  of the) Chow-Rashevski\u\i\  Theorem to  the  Kalman Theorem.  It shows that  the first implies  the second just on the basis of the following simple (but crucial) fact:  if $\cK$ contains a neighbourhood of the origin -- and thus,  in particular,  a convex symmetric neighbourhood of the $0_{\bR^n}$ --  then also  the controllable set   $\cC_{q = 0_{\bR^n}}$ of  \eqref{Thesystem-1}   contains a convex set  (see  (1) of Proposition \ref{theprop},  below).  
As we mentioned above we  expect that this new proof  admits   generalisations for other kinds of  control systems,  provided that  similar convexity properties for the controllable sets can be established.  On this regard, we would like to mention  that     investigations, involving ideas very close to those of  our first proof,    can be  found  in \cite{Kr}, where interesting extensions of the Kalman criterion are reached.\par
  Our second proof is much longer than the first, it is based on a discussion on the  higher dimensional space  $\cM = \bR \times \cQ \times \cK$ and it  does not involve in a direct way the   Chow-Rashevski\u\i\  Theorem -- it   just uses    a few  results on flows, which  are  properties that can be taken as  ancestors of  the   Chow-Rashevski\u\i\  Theorem.   However, these  drawbacks are balanced by the fact that our  second proof makes 
no  use of  any information on  convexity properties of  the controllable  sets. This  makes  this second proof 
even more appropriate than the first  for  generalisations  to  non-linear  control systems.  Indeed,  such generalisations  have been established in \cite{GSZ}, where  new  local controllability criteria are given  and diverse  controllability problems, for which all previously known criteria are inconclusive, are solved.
\par
The paper is structured as follows. After  section \S \ref{prel},  in which we give    a few preliminaries and introduce some convenient notation,  in \S \ref{sect3} we  recall the statements of 
the Chow-Rashevski\u\i\ and the Kalman Theorem, on which our discussion is based. In \S \ref{KalmantoChow},  for the sake of completeness, we exhibit a direct  proof  that,  for the collections of $m + 1$ vector fields $X_\a$ of   the above described form, the  property  determined   by  the Chow-Rashevski\u\i\ Theorem can be inferred directly from the Kalman Theorem (as we mentioned above,   other proofs that hold  in a much more general setting   are given   in   \cite{El, ST}).  In \S \ref{ChowtoKalman} and \S \ref{sect5}  there are  our main results, i.e.    the mentioned   new proofs for the Kalman Theorem, one based on the Chow-Rashevski\u\i\ Theorem and  the other on general properties  of flows. \par
Throughout this paper we  adopt the  Einstein convention on summations. \par
 \noindent{\it Acknowledgments.}
We are   sincerely grateful to David Chillingworth  for  kindly pointing out Peter \v Stefan's  works   to us.\par
\medskip
\section{Preliminaries}\label{prel}
\subsection{Vector fields, Lie brackets and flows}\  We recall that a    {\it $\cC^k$ vector field  on    $\cU \subset \bR^n$}  is   an application   in  $\cC^k(\cU, \bR^n)$, which maps  each   $x = (x^i)_{ 1 \leq i \leq n} \in \cU$  into a  vector $X(x) = (X^1(x), \ldots, X^n(x))$ applied to the  point  $x$. We also recall that  each $\cC^k$ vector field $X = (X^i)_{1 \leq i \leq n}$ can be  identified  with  the first order  differential operator on   real  functions 
$ X(f)  \= X^i \frac{\p f}{\p x^i}$.  
 Under this identification,  the   Lie bracket $[X, Y]$ between two  $\cC^k$ vector fields $X, Y$ on $\cU$, $k \geq 1$,  can be  defined as    the   $\cC^{k-1}$ vector field on $\cU$ identified with the first order  differential operator 
\begin{multline*} [X, Y](f) \= X(Y(f)) - Y(X(f)) = \left( X^j \frac{\p Y^i } {\p x^j}- Y^j \frac{\p X^i } {\p x^j}\right) \frac{\p f}{\p x^i} =  \\
 = \left( X(Y^i)- Y(X^i)\right) \frac{\p f}{\p x^i} \ .\end{multline*}
 Given  a $\cC^k$ vector field $X$ on $\cU \subset \bR^n$,   a corresponding  (local)  {\it flow}  on a  relatively compact open subset  $\wt \cU \subset \cU$  is a 
 map  $\Phi^X: (- \ve , \ve) \times \wt \cU \longrightarrow \bR^n$, which is constructed as follows  for a sufficiently small $\ve> 0$  .  For  any $ x \in \wt \cU$, let us denote by   $\g^{(x)}(s)$, $s \in (- \ve^{(x)}_1, \ve^{(x)}_2)$,  the unique maximal solution  to the 
 differential problem 
 $\dot \g^{(x)} (s) = X(\g^{(x)}(s))$, $ \g^{(x)} (0) = x$.  Given  $0 < \ve  \leq \min_{q \in \wt \cU}\{\ve^{(x)}_1, \ve^{(x)}_2\}$,  the flow $\Phi^X$  is  the map defined  by 
 $$\Phi^X: (- \ve, \ve) \times \wt \cU \longrightarrow \bR^n\ ,\qquad \Phi^X(s, x) \= \g^{(x)}(s)\ . $$
 By construction,  for any  fixed  $s_o \in (- \ve, \ve)$  the  map 
 $\Phi^X_{s_o} \= \Phi^X(s_o, \cdot):\wt  \cU \to \bR^n$
 is a diffeomorphism from $\wt \cU$ onto $\Phi^X_{s_o}(\wt \cU)$, with   inverse   $\Phi^X_{-s_o} = \Phi^{-X}_{s_o}$, and  with
 $$ \frac{\p \Phi^X(s, x)}{\p s}\bigg|_{s = 0} = X(x)\qquad \text{for any}\ x \in \wt \cU\ .$$
   Moreover,   $\Phi^X_{s_o = 0} = \Id_{\wt \cU}$ and, whenever $s,s' $ are such that $s + s' \in (- \ve, \ve)$ and $ \Phi^X_{s}(x) , \Phi^X_{s'}(x) \in \wt \cU$,   
   $$\Phi^X_{s+ s'} (x) = \Phi^X_s \big( \Phi^X_{s'} (x)\big) =  \Phi^X_{s'} \big( \Phi^X_{s} (x)\big)\ .$$
\par
\medskip
 \subsection{Pushed-forward vector fields and their flows} Let $X$ be a $\cC^k$ vector field  on an open set $\cU \subset \bR^n$ and $\f: \cU \to \cV =  \f(\cU) \subset \bR^n$ a diffeomorphism between $\cU$ and 
 $\cV \subset \bR^n$. The {\it push-forward of $X$ by the map $\f$} is the vector field  on $\cV$
 $$ \f_*(X) =\big( \f_*(X)^1, \ldots  \f_*(X)^n\big) $$ 
 corresponding to the linear differential operator defined by  
 \begin{multline} \label{uno} \f_*(X)(f)\big|_q\= X(f \circ \f)\big|_{\f^{-1}(q)} 
 = X^j(\f^{-1}(q))\frac{\p( f \circ \f)}{\p q^j} \bigg|_{\f^{-1}(q)}=  \\
 = X^j(\f^{-1}(q)) \frac{\p \f^i}{\p q^j}\bigg|_{\f^{-1}(q)} \frac{\p f}{\p q^i}\bigg|_q 
  = \left( J(\f)^i_j X^j\right)\bigg|_{\f^{-1}(q)} \frac{\p f}{\p q^i}\bigg|_q\ ,\end{multline}
  where $J(\f)$ is the Jacobian matrix of   $\f$. 
  In other words,  $\f_*(X)$ is   the vector field on $\cV$ with  components $ \f_*(X)^i(q) =   \bigg( J(\f)^i_j X^j\bigg)(\f^{-1}(q))$.\par 
 The following    fact is well known (see e.g. \cite[Prop. 1.55]{Wa}):
 {\it  For any two $\cC^k$ vector fields $X$, $Y$, $k \geq 1$, on $\cU$  and any $\cC^2$ diffeomorphism $\f: \cU \to \cV \subset \bR^n$ then   $\f_*([X, Y]) = \left[\f_*(X), \f_*(Y)\right]$}.  This and a  direct calculation imply that  the  following relation between flows  holds at all  points where both sides are defined: 
  \beq  \label{useful} 
 \Phi^X_s \circ \Phi^Y_t = \Phi_t^{(\Phi^X_{s})_*(Y)}  \circ \Phi^{X}_{s}\ .
  \eeq
 \par 
 \subsection{Piecewise regular curves and their compositions}
 \label{2_3}
Given a manifold $\cN$,   by   {\it  parameterised regular curve} we mean a  $\cC^1$ map $\g:[a,b] \subset \bR  \longrightarrow \cN$    with nowhere vanishing velocity $\dot \g(t)$ and which is  an homeomorphism between  $[a,b]$ and  the  image $\g([a,b])$ of the map.  The image $\g([a,b]) \subset \cN$ is called {\it (non-parameterised) regular curve}.   
     Two  parameterised regular curves  $\g(t)$,  $\wt \g(s)$ with the  same image  are called  {\it consonant} if  one is determined    from the other via  a   change of  parameter $t = t(s)$ with 
   $\frac{dt}{ds} > 0$ at all points. Consonance  is clearly an equivalence relation  and    an  {\it orientation} of a  regular curve  is  a  choice of one of the   only two possible   equivalence classes of its  parameterisations. 
 An {\it oriented curve} is a regular curve  with an orientation. In the following,  we denote it  by means of  one  of the consonant   parameterisations $\g(t)$ of the chosen equivalence class.
  \par
\smallskip
Given two    oriented curves  $\g_1(t)$, $\g_2(s)$, $t \in [a, b]$,  $s \in [a',b']$,  with   $\g_1(b) = \g_2(a')$, their union  is called     {\it (oriented) composition} and we denote it by  $\g_1\ast \g_2$. The  oriented regular curves  $\g_1$, $\g_2$  that  give  the composition   $\g_1 \ast \g_2$ are called  {\it regular arcs} of the composition.  In a  similar  way  we define the   {\it (oriented) composition} of a finite number of oriented regular curves,  each of them sharing its final endpoint with the initial endpoint of the succeeding one.  The   subsets of $\cN$, which are   compositions of a finite number  of oriented regular curves,   are called  {\it piecewise regular (oriented) curves}.\par
\medskip
\section{The  Chow-Rashevski\u\i\ Theorem and the Kalman Theorem}\label{sect3}
 \subsection{Orbits of sets of vector fields and the  Chow-Rashevski\u\i\ Theorem} 
 Let  $\cF = \{ X_A\}_{A \in \cJ}$  be a family of $\cC^\infty$ vector fields on an open subset  $\cU \subset \bR^n$, indexed by the elements of a (possibly uncountable) set $\cJ$,  and denote by $\wh \cF$  the space of  all finite linear combinations with constant coefficients  of the vectors in $\cF$, i.e. 
 $$\wh \cF = \langle \cF \rangle \= \bigg\{Y = \l_{A_1} X_{A_1} + \ldots +   \l_{A_N} X_{A_N} \ , \  \l_{A_i} \in \bR\ ,\ N\in \bN\bigg\}\ .$$
The  {\it orbit   by $\cF$ of a point $x_o \in \bR^n$} 
 is the subset  of $\bR^n$ 
 \begin{multline} \orb^\cF(x_o) = \bigg\{\ x \in \cU\ : \ x = \Phi^{Y_{1}}_{s_1} \circ  \ldots \circ \Phi^{Y_{N}}_{s_N}(x_o)\ ,\ \  N \in \bN\ , \ s_i \in \bR \ , \ Y_{i} \in \wh \cF\ \bigg\}\ .
 \end{multline}
 The {\it Lie  span of $\cF$  at  the point $x_o$} is the linear subspace  of $T_{x_o} \bR^n = \bR^n$ defined by  
$$\Lie(\cF)_{x_o} \= \bigg\{\ v = [Y_1, [Y_2, \ldots [Y_{N - 2}, [Y_{N-1}, Y_N]]\ldots]]\big|_{x_o}\ ,\  N \in \bN\ ,\  Y_j \in \wh \cF\bigg\}\ .$$
This notation being fixed,  we can state Sussmann's refined and improved  version of the  Chow-Rashevski\u\i\ Theorem as follows (see also \cite{St1, St2} for an even  more general variant of Sussmann's version and  \cite{GSZ1, GSZ} for discussions and  improvements in various directions). 
\begin{theo}[Chow-Rashevski\u\i-Sussmann \cite{Ch, Ra, Su0, Su}] \label{Chow_Rashevsky}  For any  family $\cF$ of $\cC^\infty$ vector fields on an open subset $\cU \subset \bR^n$ and for any $x_o \in \cU$, the corresponding set  $\orb^\cF(x_o)$ is  a connected  immersed submanifold of $\bR^n$ of dimension greater than or equal to  $\dim \Lie(\cF)_{ x_o }$. 
Moreover, if the vector fields are real analytic, then     $\orb^\cF(x_o)$  contains  a   neighbourhood of $x_o$
 if and only if    $\dim  \Lie(\cF)_{x_o} = n$.
\end{theo}
\par
 \medskip
 \subsection{The Kalman Theorem}\label{3322}
Consider the   system \eqref{Thesystem-1}, which,  for brevity,  in the following will be   called {\it the linear system} $(A,B)$.  Given a  measurable control $u_o:  [0, T] \to \cK$,  we  denote by  $q^{(A, B| q_o, u_o)}:[0, T] \to \cQ$ the 
unique absolutely continuous map which  solves the equations with control $u_o$ and initial condition $q_o$,  that is  the solution to 
\beq \label{Thesystem-1*} \begin{array}{ll} \dot q(t)  = A q(t)  + B u_o(t)\  \text{for a.e.}\ t \in [0, T]\ ,
\\ q(0) = q_o\ .
\end{array}\eeq
For any choice of points $\overline q, q_o \in \cQ$ we may  consider  the following associated subsets of $\cQ$: 
\begin{itemize}[leftmargin = 15pt]
\item[--] The  {\it set of the  states  that are controllable to $\overline q$}, i.e. the set of $q \in \cQ$  from which the controlled system can start  to   reach the prescribed {\it final  state} $\overline q$ in a finite time $T \geq 0$;  
\item[--] The {\it set of all states  that are reachable from  $q_o$},  i.e.  the set of  $q \in \cQ$  that can be  reached in finite time $T 	\geq 0$ starting from  the prescribed {\it initial state} $q_o$. 
\end{itemize}
We denote such two sets, respectively,   by 
\beq \label{C} \cC^{(A, B)}_{\overline q} \= \bigg\{ q_o \in \bR^n\ : 
\ q^{(A,B|q_o, u_o)}(T) = \overline q\ , \ 
 \text{for some} \ T\geq  0\  \text{and} \ u_o:[0, T] \to \cK\ \text{meas.}\bigg\}\ ,\eeq
\beq \label{O} \cO^{(A, B)}_{q_o} \=  \bigg\{ \overline q \in \bR^n\ :
  \ q^{(A,B| q_o, u_o)}(T) = \overline q\ ,\
\text{for some} \ T\geq  0 \  \text{and} \ u_o:[0, T] \to \cK\ \text{meas.}\bigg\}\ .\eeq
The  subsets  of these two sets,  given by the points that are controllable  to $\bar q$ or reachable from $q_o$, respectively,     by means of  piecewise constant controls are denoted by   $\cC^{(A, B)\text{p.c.}}_{\overline q} $ and   $\cO^{(A, B)\text{p.c.}}_{q_o} $, respectively.\par
 \smallskip
  Due to the linearity of the system \eqref{Thesystem-1} and the fact that the matrices $A$ and $B$ are constant, 
  the set of controllable states and the set of reachable states are  related each other  as follows.  If we  set  $\bA \= - A$, $\bB = - B$,  one can 
  directly see that  for any $q \in \cQ = \bR^n$, 
 \beq \label{3_6}  \cO^{(A, B)}_{q} = \cC^{(\bA, \bB)}_{q} \ ,\qquad  \cC^{(A, B)}_{q} = \cO^{(\bA, \bB)}_{q}\ .   \eeq 
Moreover, if the control set $\cK$ is symmetric (i.e. such that $\cK = - \cK$) we also have that 
   \beq  \label{3_7}  \cC^{(A, B)}_{q} = \cC^{(A, \bB)}_{q} = \cO^{(\bA, B)}_{q}  \ ,\qquad  \cO^{(A, B)}_{q} = \cO^{(A, \bB)}_{q} = \cC^{(\bA, B)}_{q}\ . \eeq 
 Similar relations hold for the  sets $\cC^{(A, B)\text{p.c.}}_{q} $, $\cO^{(\bA, B)\text{p.c.}}_{q} $, etc.
\par
 We are now ready to state the version   of the Kalman  Theorem,  which we are going to use in all forthcoming discussion  and which is well known to be fully equivalent to the statement  in the Introduction. 
In a matter of clarity,   in what follows  the   space $V$,  generated by the vectors defined in  \eqref{tre},  will be   denoted by $V = \Kalman^{(A, B)}$.
  \begin{theo}[Kalman Theorem] \label{Kalman-criterion} If  $\cK\subset\bR^m$ is a neighbourhood of $0_{\bR^m}\in\bR^m$, the set   $\cC^{(A,B)\text{\rm p.c.}}_{\overline q = 0}$  contains a neighbourhood of $ 0_{\bR^n}\in \bR^n$ if and only if   $\dim \Kalman^{(A, B)} = n$. 
 \end{theo}
The  equality $\dim \Kalman^{(A, B)} = n$ is commonly called {\it Kalman condition}.
\medskip
  \section{The  Kalman Theorem implies the  corollary of the Chow-Rashevski\u\i\ Theorem for sets of 
  vector fields with  either constant or linear components}\label{KalmantoChow}
Given a linear system $(A, B)$, let   $\cF^{(A, B)}$ 
be     the finite  set of   real analytic vector fields  $\cF^{(A, B)} \= \{X_0, X_1, \ldots, X_m\}$ on $\bR^n$,  with  $X_\a$   defined at each point  $q$  by 
\beq  \label{211}
\begin{split}
&\hskip 4 cm X_0(q)  \= (-A^1_j q^j, \ldots, -A^n_j q^j) =- A\,q\qquad \text{and} \\
 &X_1(q) \= (B_1^1, \ldots, B_1^n)\ ,\ X_2(q) \= (B_2^1, \ldots, B_2^n)\ ,\ \ldots, \ X_m(q) \= (B_m^1, \ldots, B_m^n) \ .
  \end{split}
  \eeq
  Note that, for $\a=1,\ldots,m$, the components of the  vector field $X_\a$  are precisely  the entries of the $\a-$th column of the matrix $B$.
  Finally, for  any  $\l = (\l^i) \in \bR^m$, let us set
  $Y^{(\l)} \= X_0 + \l^i X_i$. \par
 We want  to   show  that  {\it using  the  Kalman Theorem in a direct and  elementary way, one can  prove that  the orbit   $\orb^ {\cF^{(A,B)}}(0_{\bR^n})$ contains a neighbourhood of the origin if and only if  $\dim \Lie(\cF^{(A,B)})_{q = 0}=n$}. 
  \par
 Indeed, 
   for any   $\wh q \in \bR^n$ and   $\l \in \bR^m$,  the   integral  curve $t \mapsto \Phi^{Y^{(\l)}}_t(\wh q)$  of the vector field $Y^{(\l)}$ coincides with the solution   $q^{(-A, B| \wh q, u_o)}(t)$ of the linear system $(-A, B)$ determined by the constant control $u_o(t) \equiv \l$. Hence if we consider the linear system $(A, B)$  with a  control set  $\cK$ which  is  not only  open, but also  symmetric,  by   \eqref{3_6} and \eqref{3_7}  for any fixed point $\overline q$,   the corresponding set $\cC^{(A, B)\text{p.c.}}_{\overline q} \bigg(= \cO^{(-A, B)\text{p.c.}}_{\overline q}\bigg)$  coincides with  the set of points  of  the form
  $$q(T) = \Phi^{Y^{(\l_N)}}_{s_N} \circ \ldots \circ \Phi^{Y^{(\l_1)}}_{s_1}(\overline q)\ ,\qquad T = s_1 + s_2 + \ldots + s_N> 0\ ,$$
 for some finite sequence of  $m$-tuples $\l_i = (\l_i^1, \ldots, \l^m_i)$.  Since each such a  point is  in $\orb^{\cF^{(A, B)}}(\overline q)$,   we conclude that 
$\cC_{\overline q}^{(A, B) \text{\rm p.c.}}  \subset \orb^ {\cF^{(A,B)}}(\overline q)$. 
  From this it follows immediately that {\it if the linear system $(A, B)$ satisfies the Kalman condition, then, by the Kalman Theorem, not only      $\cC_{\overline q =0_{\bR^n}}^{(A, B) \text{\rm p.c.}}$ contains a neighbourhood of $0_{\bR^n}$,  but      also   the orbit $\orb^ {\cF^{(A,B)}}(0_{\bR^n})$ contains a neighbourhood of the origin} as  desired. \par
  \smallskip
We claim that the converse is  true as well.   The proof of  this   is indeed elementary and   needs neither the Kalman Theorem nor other results nor convexity assumptions,  except some standard  linear algebra. 
For the sake of completeness,  here are the  details.\par
Assume that the linear system $(A, B)$ does not satisfy the Kalman condition, i.e. that the vector space $\Kalman^{(A, B)}$   has dimension $n' < n$, and  pick a set of  $n'$ linearly independent vectors among  the generators \eqref{tre}, say $C_1, \ldots, C_{n'}$.  Then consider a set of $n - n'$ additional vectors, denoted by $C_{n'+1}$, \ldots, $C_n$ not belonging to $\Kalman^{(A, B)}$ ,  which  complete the previous vectors  to a basis for $\bR^n$.  
 Applying an appropriate linear change of  coordinates, there is no loss of generality if we assume that  $C_1 = (1, 0, \ldots, 0)$, $C_2 = (0, 1, \ldots, 0)$, \ldots, $C_n = (0, 0, \ldots, 1)$  and that 
  $$\Kalman^{(A, B)} = \langle C_1, \ldots, C_{n'} \rangle =  \{q_{n' + 1} =  \ldots = q_n = 0\}. $$
Note that,  being  each vector  $C_i$, $1 \leq i \leq n'$,   in $\Kalman^{(A, B)}$,  the corresponding  vector $A \cdot C_i$ is  also   in $\Kalman^{(A, B)}$. This implies that  the matrix $A$ (in the considered new coordinates!) has the form 
\beq A = \left( \begin{array}{cccccc} a_1^1 & \hdots & a_{n'}^1 & a_{n' + 1}^1 & \hdots & a_n^1\\
a_1^2 & \hdots & a_{n'}^2 & a_{n' + 1}^2 & \hdots & a_n^2\\
\vdots & \ddots & \vdots  & \vdots  & \ddots & \vdots\\
a_1^{n'} & \hdots & a_{n'}^{n'} & a_{n' + 1}^{n'} & \hdots & a_n^{n'}\\
0& \hdots & 0 & a_{n' + 1}^{n'+1}  & \hdots & a_n^{n'+1}\\
\vdots & \ddots & \vdots  & \vdots  & \ddots & \vdots\\
0& \hdots & 0 & a_{n' + 1}^{n}  & \hdots & a_n^{n}
\end{array} 
\right)\ .\label{specialA} \eeq
 In particular,  for any $q \in \Kalman^{(A, B)}$, also the vector $A {\cdot} q$ is in $\Kalman^{(A, B)}$.  It follows that the value  at any  point of $\Kalman^{(A, B)} = \{q_{n' + 1} =  \ldots = q_n = 0\}$ of a vector field in $\cF^{(A, B)}$   is a  vector which is tangent to $\Kalman^{(A, B)}$ and  that 
 any  flow  of a  vector field in $\wh{\cF^{(A,B)}}$ maps the subspace $\Kalman^{(A, B)}$ into  itself. We conclude that {\it if the linear system $(A, B)$ does not satisfies the Kalman condition, then  the orbit $\orb^ {\cF^{(A,B)}}(0_{\bR^n})$ is contained  in $\Kalman^{(A, B)} \subsetneq \bR^n$ and cannot contain any neighbourhood of the origin}, as claimed. 
  \par
  \smallskip
  In order to conclude the proof of the claim   \\[3pt]
  ``{\it $\orb^ {\cF^{(A,B)}}(0_{\bR^n})$  contains a neighbourhood of $0_{\bR^n}$ if and only if   $\dim  \Lie(\cF^{(A,B)})_{q = 0} = n$}'' \\[3 pt] 
   it remains   to show that,  for such a kind of set of vector fields,  the Kalman condition is equivalent to the condition  $\dim \Lie(\cF^{(A, B)})_{ q = 0} = n$. For this,  it suffices to notice that,  due to  the constancy of the vector fields $X_1,\ldots, X_m$,  in  the set of all    iterated Lie brackets $[X_{i_1}, [X_{i_2}, \ldots, [X_{i_{N-1}}, X_{i_N}]\ldots]]$, $N \geq 2$,  the only non   trivial  brackets  are those having the form
 \beq \label{Lielie}
\underset{\text{\rm $N-1$-times}}{\underbrace{ [X_{0}, [X_{0}, \ldots, [X_{0}, }}X_{j}]\ldots]] =  \big(\big(\underset{\text{\rm $N-1$-times}}{\underbrace{A {\cdot} \ldots {\cdot} A}}\big)^1_\ell B_j^\ell, \ldots, \big(\underset{\text{\rm $N-1$-times}}{\underbrace{A {\cdot} \ldots {\cdot} A}}\big)^n_\ell   B^\ell_j \big)\ . \eeq
From this  it follows immediately that $ \Kalman^{(A, B)}  =  \Lie(\cF^{(A, B)})\big|_{\overline q = 0}$ and this   immediately gives  the desired equality  between the two dimensions.
\par
\bigskip
 \section{A proof of the Kalman Theorem,
 based on the  Chow-Rashewski\u\i\ Theorem} \label{ChowtoKalman}
We give now our first new proof of the Kalman Theorem,  which shows how that theorem can be obtained as a corollary of the  Chow-Rashevki\u\i\ Theorem. \par
\smallskip
First of all, we stress the fact that the necessity part of  Theorem \ref{Kalman-criterion} (or, equivalently, the fact that if the Kalman condition does not hold, then the set  $\cC^{(A,B)}_{\overline q = 0_{\bR^n}} = \cO^{(-A,B)}_{\overline q = 0_{\bR^n}}$ cannot contain any neighbourhood of the origin) is a consequence of  the same linear algebra  arguments   considered at the end of \S \ref{KalmantoChow}. Indeed, if $\dim \Kalman^{(A,B)} = n'  < n$, by  those arguments,  there  is no loss of generality if we  assume that   $ \Kalman^{(A, B)}$ is the space $ \Kalman^{(A, B)} =  \{ q^{n' + 1} = \ldots = q^N = 0\}$, that    $A$ has the form \eqref{specialA} and  that  all vectors $B_\a$, $1 \leq \a \leq m$, are in the subspace $\Kalman^{(A, B)}$.  If this is the case, one can immediately see that  any solution of the linear system $(-A, B)$ that  starts from $0_{\bR^n}$  is constrained   to stay in  $ \Kalman^{(A,B)} $, so that $ \cO^{(-A,B)}_{\overline q = 0_{\bR^n}}= \cC^{(A,B)}_{\overline q = 0_{\bR^n}} \subset  \Kalman^{(A,B)} $ does not  contain any neighbourhood of the origin. \par
\smallskip
Second,  we point out that, in order to  prove the sufficiency part of the Kalman Theorem, there is  no loss of generality if we  assume that $\cK$ is not only a neighbourhood of $0_{\bR^n}$, but  is also 
convex and symmetric.  Indeed,      if we use the notation  $\cC^{(A,B)\text{\rm p.c.}}_{\overline q = 0_{\bR^n}}$   and   $\wt{\cC^{(A,B)\text{\rm p.c.}}}_{\overline q = 0_{\bR^n}}$ for   the sets of points that are controllable to $0_{\bR^m}$  by means of controls in   $\cK$ and of controls  in a   subset $\wt \cK$ of  $ \cK$, respectively, it is  clear that    $\wt{\cC^{(A,B)\text{\rm p.c.}}}_{\overline q = 0_{\bR^n}} \subset  \cC^{(A,B)\text{\rm p.c.}}_{\overline q = 0_{\bR^n}}$.   If in addition $\cK$  contains  a neighbourhood of $0_{\bR^m}$, then it  certainly  contains a  subset $\wt \cK$ which is not only a neighbourhood of $0_{\bR^m}$, but it also  convex and symmetric (a ball, for example).  Hence if we  prove   that the Kalman condition  is sufficient  for   $\wt{\cC^{(A,B)\text{\rm p.c.}}}_{\overline q = 0_{\bR^n}}$  to  contain a neighbourhood of the origin,  we  immediately  get  that the Kalman condition   a fortiori  is a sufficient condition for  $\cC^{(A,B)\text{\rm p.c.}}_{\overline q = 0_{\bR^n}}$ having  that  property.\par
\smallskip
Due  this observation,   the proof  of the sufficiency part of the Kalman Theorem can be made  under the stronger assumption that $\cK$ is open, convex,  symmetric and contains the origin.   
For such a proof we need the following 
\begin{prop} \label{theprop}  If $\cK \subset \bR^m$ is  open,  convex,  symmetric  and contains the origin, then:
\begin{itemize}[leftmargin = 20pt]
\item[(1)]  $\cC_{\overline q = 0}^{(A,B)p.c.}$  is  symmetric and convex; 
\item[(2)] If $L \subset \bR^n$ is  the smallest  (with respect to inclusion)   linear subspace containing  $\cC_{\overline q = 0}^{(A,B)p.c.}$,  
then $\dim L\geq 1$  and   $\cC_{\overline q = 0}^{(A,B)p.c.}$ contains a neighbourhood of $0$  of the   topology of  $L$ (considered as a topological subspace of $\bR^n$); 
\item[(3)]   $\dim L = \dim \orb^{\cF^{(A, B)}}(0) \geq   \dim  \Lie(\cF^{(A, B)})_{\overline q=0}$.
\end{itemize}
\end{prop}
\begin{proof} 
 (1)   is a well known property of linear system, which   can be  easily  checked by just considering  the classical  integral representation formula for the solutions to   \eqref{Thesystem-1} (see  e.g. \cite[\S II.2]{MS}).\par
In order to check  (2),  note that, since $B \neq 0$,  the  set  $\cC_{\overline q = 0}^{(A,B)p.c.}$ contains at least a point different from $0$. This implies that    $\dim L \geq 1$.  Denoting  $N = \dim L$, consider   $N$  points $q_1, \ldots, q_N \in \cC_{\overline q = 0}^{(A,B)p.c.}$  that are linearly independent (they exist because, otherwise,    $L$ would not be  the smallest linear space  containing $ \cC_{\overline q = 0}^{(A,B)p.c.}$) and let $C$ be  the convex hull of the set $\{ q_1, -q_1, q_2, -q_2, \ldots, q_N, - q_N\}$.  The interior of $C$ in the topology of $L$ is  an open neighbourhood of the origin and it  is contained in $\cC_{\overline q = 0}^{(A,B)p.c.}$ because of (1). \par
We now prove  (3). First of all we claim that each vector field $X_\a$,    defined in \eqref{211}, is  tangent to  $L$ at all  points of  $\cC^{(A, B)p.c.}_{\overline q = 0} $.
In order to check this, for what concerns a    vector field $X_{\a_o}$, $\a_o \neq 0$, suppose that the claim is not true, i.e.  that $X_{\a_o}|_{q} \notin L$ for some $q \in \cC^{(A, B)p.c.}_{\overline q = 0}$. Since $X_{\a_o}$ has constant components, this implies that also  $X_{\a_o}|_{\overline q = 0} \notin L$.  In this case,  consider   the differential problem   $\dot q^i = - A^i_j q^j - \mu B_{\a_o}$, $q(0) = 0$ for some small constant value $\mu>0$ and let  $q(t)$ be a solution to this problem, 
hence a curve, which is   tangent to the vector field  $X_0(q) - \mu  X_{\a_o}(q)$.  Note that, in the linear system $(-A,B)$, this corresponds to the choice of the constant control 
$\mu e_{\alpha_o}$ which, for small $\mu>0$, belongs to $\cK$ because it is a neighbourhood of the origin in $\bR^m$, being $e_{\alpha_o}$ the suitable element of its canonical basis.
Since  $X_{\a_o}|_{\overline q = 0} \notin L$ and  $X_0|_{\overline q = 0} = 0$,  for small values  of $t$ the point  $q(t)$   is not in $L$.  This is however impossible because  all points of such a solution  $q(t)$  are in $\cO^{(-A, -B)p.c.}_{\overline q = 0} = \cC^{(A,B)p.c.}_{\overline q = 0} \subset  L$. For what concerns the vector fields $X_0$ (which has not constant components), the proof is  different. As before, assume that the claim does not hold for $X_0$, i.e.  that there is a  $\overline q'  \in \cC^{(A,B)p.c.}_{\overline q = 0}  $ such that $X_0|_{\overline q'}$ is not tangent to $L$.  We may therefore consider 
\begin{itemize}[leftmargin = 20pt]
\item[(i)]  A solution $q(t)$ of a differential problem  $\dot q = - A q - B u(t)$, $q(0) = 0$,  that ends at $\overline q'$  and  is determined by a piecewise constant control $u(t)$ (it exists because  $\overline q'  \in  \cC^{(A,B)p.c.}_{\overline q = 0}   =  \cO^{(-A, -B)p.c.}_{\overline q = 0}$); 
\item[(ii)] A  solution $q(t)$ to the differential problem $\dot q = - A q$, $q(0) = \overline q'$ (such a  solution is an integral curve of the vector field $X_0$ and, since $X_0|_{\overline q'} \notin L$,  the points $q(t)$ are not  in $L$ for small values of  $t$; and again, it is an admissible curve for the controlled system $(-A,B)$ because $0\in\cK$).
\end{itemize}
 Composing  the  curves in (i) and (ii), we get a  curve  which is entirely included in    $ \cO^{(-A, -B)p.c.}_{\overline q = 0} = \cC^{(A,B)p.c.}_{\overline q = 0} $, but also which is not entirely included in $L$ contradicting the assumption  $\cC^{(A,B)p.c.}_{\overline q = 0}  \subset  L$.\par
 \smallskip
From  the argument above,   we see that all vector fields in the set $\cF^{(A, B)}$ and, consequently,  all vector fields  in the spanned vector space  $\wh{\cF^{(A, B)}}$ are tangent to $L$ at the points of an open  (in the topology of $L$) neighbourhood $\cU \subset L$ of $0$,   which is entirely contained  in $\cC^{(A,B)p.c.}_{\overline q = 0} $ (it exists by (2)).Thus  any integral curve $\sigma^X(s)$ of a vector field  $X$ in  $\wh{\cF^{(A, B)}}$  that starts from a point in $\cU$ stays in $\cU$ for small values of $s$.
Consequently,  for any $q \in \cU \subset  \cC^{(A,B)p.c.}_{\overline q = 0} $,  there exists an open neighbourhood of $q$ of the intrinsic topology of  $\orb^{\cF^{(A, B)}}(q)  $ which  is included in  $\cU \subset L$.   By  Theorem \ref{Chow_Rashevsky},  $\orb^{\cF^{(A, B)}}(\overline q_o = 0)$ is  an immersed submanifold of $\bR^n$ of dimension greater than or equal to  $\dim \Lie(\cF^{(A, B)}_{\overline q_o = 0})$. Thus: 
\begin{itemize}[leftmargin = 20pt]
\item   $\dim  \orb^{\cF^{(A, B)}}( q_o = 0) \leq \dim L$ because  $\orb^{\cF^{(A, B)}}( q_o = 0)$ contains an open subset of its intrinsic topology included   in $\cU \subset L $; 
\item   $\dim  \orb^{\cF^{(A, B)}}( q_o = 0) \geq \dim L$ because $\orb^\cF( q_o = 0)$  contains  $\cC^{(A,B)p.c.}_{\overline q = 0}$ which contains an open subset of  $L$. 
\end{itemize}
This  implies that $\dim L = \dim \orb^{\cF^{(A, B)}}(\overline q_o = 0) \geq \dim  \Lie(\cF^{(A,B)})_{\overline q_o = 0}$.
\end{proof}
We have now all ingredients to conclude the proof of the Kalman Theorem described at the beginning of this section.  By
 \eqref{Lielie} and the remarks just before that equality,   we know that  
$ \dim \Kalman^{(A,B)} = n$ if and only if  $  \dim  \Lie(\cF^{(A,B)})_{\overline q_o = 0} = n$. Thus, by  Proposition \ref{theprop} (3),   the linear subspace $L \subset \bR^n$ of claim  (2)  of that proposition    is equal to $\bR^n$ and  the set  $\cC^{(A,B)p.c.}_{\overline q = 0}$ is   a symmetric,  convex neighbourhood of $\overline q=0$ in $\bR^n$. In particular it is  an open subset of   $L = \bR^n$.
 Thus {\it if the Kalman condition holds,  then    $\cC^{(A,B)p.c.}_{\overline q = 0}$ contains a neighbourhood of $0_{\bR^n}$} as we needed to prove. \par
\begin{rem}
We  point out that a generic  problem of controllability is a ``forward-in-time problem'' in the following sense:  It consists of the query of  points  that can be considered as starting points for the  controlled evolutions that  reach a desired destination in  a  finite  time {\it and moving just forward in time}. On the other hand,  the points   of  an orbit of $\cF^{(A, B)}$, as  considered in the Chow-Rashevski\u\i\ Theorem,   are points, which can be   reached moving {\it both forward and backward in time} along the evolutions. This is the crucial difference between the two problems. In the above proof, the argument  where the``moving-just-forward-in-time'' constraint on the evolutions  (i.e. the main difference  between the  controllability problems and the  Chow-Rashevski\u\i\  type problems)  is shown to be non effective for deriving the Kalman Theorem  is   when we  prove that all the vector fields $X_\alpha$ (and in particular, the vector field  $X_0$)  are tangent to $L$.  Indeed this is equivalent to the  fact that $C_{\overline q=0}^{(A,B)p.c.}$ is a neighbourhood of $0$ if and only if $C_{\overline q=0}^{(-A,B)p.c.}$ is so and  hence   if and only if  $C_{\overline q=0}^{(\pm A,B)p.c.}$ is so as well. Note that $C_{\overline q=0}^{(\pm A,B)p.c.}$ is precisely the set of the controllable points of the problem given by the controlled dynamics $\dot q=\sigma Aq+Bu$ with  controlling parameter  $u = (u^i)$ together with  the additional discrete parameter  $\sigma\in\{-1,1\}$.
\end{rem}

\par
\medskip
 \section{Another  proof of  the Kalman Theorem,
 based on  general facts on   flows   of  vector fields}
 \label{sect5}
We now give our second proof of  the Kalman Theorem or, more precisely, only of the sufficiency part of that theorem. In fact we already pointed out that   the necessity part   is just  a consequence of  
a  linear algebra argument,  namely the same that proves the necessity part in the corollary  of the Chow-Rashevski\u\i\ Theorem that we are considering  in this paper. \par
This second proof is based on a translation of the control problem given  by a linear system  $(A, B)$ into a corresponding  equivalent problem on certain  distinguished curves and  flows of certain vector fields in the extended space-time $\cM = \bR \times \cQ \times \cK$. We introduce these special curves and vector fields  and the above described translation in the next  subsections \S \ref{section6.1} and \S \ref{section6.2}. The actual proof is  in the concluding subsection  \S \ref{section6.3}.\par
   \subsection{Stepped  graphs in the extended space-time}\label{section6.1} Assume that $\cK \subset \bR^m$ is an open neighborhood of the origin.
   Consider now the manifold  $\cM =\bR \times \cQ \times \cK =  \bR^{1 + n } \times \cK$ (which is an open subset of  $\bR^{1  +  n + m}$)  and denote by $t$, $q = (q^i)$ and $u = (u^a)$ the standard cartesian coordinates of $\bR$, $\cQ$ and $\cK \subset \bR^m$, respectively, so that   each   $x \in \cM$ is   determined by a  tuple $(t, q^i, u^a)$.  As we mentioned in the Introduction, we name $\cM$ the {\it extended space-time} associated with  the   system  \eqref{Thesystem-1} and we denote by 
   $$\pi^\bR : \cM =\bR \times \cQ \times \cK \to \bR\ ,\ \  \pi^\cQ: \cM =\bR \times \cQ \times \cK\to \cQ\ ,\ \  \pi^\cK:  \cM =\bR \times \cQ \times \cK \to \cK$$
   the three standard projections of  the cartesian product $\cM$ onto  its three factors. 
   \par
   A {\it stepped (completed)  graph of a solution of \eqref{Thesystem-1}} is a piecewise regular  oriented curve  (see \S  \ref{2_3})
   $\h = \h_1 \ast \h_2 \ast  \ldots\ast \h_{2 r + 1}$
    in  $\cM $,  obtained by composing   an odd number  of regular arcs and satisfying the following three conditions (Fig.1):
\begin{itemize}[leftmargin = 15pt]
\item[(1)]  Each {\it odd}  arc $\h_{2 k + 1}$, $0 \leq k \leq r$,  is 
the graph 
$$\h_{2 k + 1} = \{\ (t, q_k(t), u_k(t))\ ,\ t \in [t_k, t_k + \s_k]\ \}$$
of a   map   $t \longmapsto (q_k(t), u_k(t)) \in \cQ \times \cK$, $t \in [t_k, t_k +\s_k]$, in  which $u_k(t) \equiv c_{k }$ is constant and  $q_k(t)$ is the corresponding solution  to  \eqref{Thesystem-1}; we denote its    initial point by   $\wt q_{k } = q_k(t_k)$; 
\item[(2)]  For any odd regular arc $\h_{2k + 1}$ with $k \geq 1$, the  initial condition  $\wt q_k$ of  $q_k(t)$ is equal to   the final point  $\wh q_{k-1} \= q_{k -1}(t_{k - 1} +\s_{k-1})$   of the  solution  $q_{k-1}(t)$; 
\item[(3)]   Each {\it even} regular arc $\h_{2(k + 1)}$, $0 \leq k \leq r-1$,  is the  segment in $\bR \times \cQ \times \cK$ that  joins the final point of    $\h_{2 k + 1}$  to the initial  point of    $\h_{2 (k +1) + 1}$. 
\end{itemize}


\centerline{
    \begin{tikzpicture}[scale=0.7]
		\node (0) at (0, 0) {};
		\node (1) at (0, 7) {};
		\node (2) at (-3, -3) {};
		\node (3) at (3, 3) {};
		\node (4) at (10, 0) {};
		\node (5) at (-1, -2.5) {};
		\node (6) at (1.25, -1.5) {};
		\node (7) at (1.25, -0.25) {};
		\node (8) at (3.5, 0.5) {};
		\node (9) at (7.25, 2.75) {};
		\node (10) at (9.5, 4.75) {};
		\node (11) at (3.5, 1.5) {};
		\node (12) at (7.25, 3.75) {};
  \fill(5)circle(3 pt);
  \fill(6)circle(3 pt);
  \fill(7)circle(3 pt);
  \fill(8)circle(3 pt);
  \fill(9)circle(3 pt);
 \fill(10)circle(3 pt);
\fill(11)circle(3 pt);
  \fill(12)circle(3 pt);
		\node (13) at (5, 2) {};
		\node (14) at (8, 4) {};
		\node (15) at (9.5, -0.5) {\(t\)};
		\node (16) at (0.5, 6.5) {\(w\)};
		\node (17) at (-3, -2.25) {\(q\)};
		\node (18) at (-1, -3) {\((t_0,q_0,w_0)\)};
\draw (3.center) to (2.center)[->]; 
		\draw (0.center) to (1.center)[->];
		\draw (0.center) to (4.center)[->];
		\draw [in=180, out=45, looseness=0.75] (5.center) to (6.center);
		\draw [in=-135, out=45, looseness=0.75] (7.center) to (8.center);
		\draw [in=-135, out=45, looseness=0.75] (11.center) to (13.center);
		\draw [in=-180, out=45] (13.center) to (9.center);
		\draw [in=150, out=60, looseness=1.25] (12.center) to (14.center);
		\draw [in=-135, out=-30] (14.center) to (10.center);
		\draw (6.center) to (7.center)[red, densely dotted, very thick];
		\draw (8.center) to (11.center)[red, densely dotted, very thick];
		\draw (9.center) to (12.center)[red, densely dotted, very thick];
\end{tikzpicture}
}
\centerline{\tiny \bf Fig. 1 -- A stepped graph in $\cM = \bR \times \cQ \times \cK$}
Note that  any   stepped  graph $\h = \h_1 \ast \h_2 \ast  \ldots\ast \h_{2 r + 1}$  of a solution satisfies the following: 
\begin{itemize}[leftmargin = 15pt]
\item The  projection on $\bR \times \cQ$  of each  even  arc $\h_{2 (k+1)}$   is a  singleton,  i.e.  the set which contains just   the point given by  the projection onto $\bR \times \cQ$ of the final point of $\h_{2 k + 1}$ (which is also the initial point of $\h_{2 k + 3})$
  $$ \big\{(t_{k-1} + \s_{k-1},  \wh q_{k-1} )\big\} = \big\{(t_k, \wt q_k)\big\}\ ; $$ 
\item The projection of   $\h$ onto $\bR \times \cQ$   coincides with  the graph  of the solution $t \to q(t)$ to  \eqref{Thesystem-1} with initial condition $ \wt q_0$ and     the piecewise constant  control curve $u(t)$,  which is   equal to   the constants $c_0, c_1, \ldots, c_r$ along   the time  intervals  $(t_k, t_k + \s_k) \subset [0, T]$ (Fig.2); 
\centerline{
    \begin{tikzpicture}[scale=0.7]
		\node  (0) at (0, 0) {};
		\node  (2) at (0, -3) {};
		\node  (3) at (0, 3) {};
		\node  (4) at (12.5, 0) {};
		\node  (5) at (1, -1.5) {};
		\node  (6) at (3.25, -0.5) {};
		\node  (7) at (3.25, -0.5) {};
		\node  (8) at (5.5, 0.25) {};
		\node  (9) at (9.25, 1.5) {};
		\node  (10) at (11.5, 2.5) {};
		\node  (11) at (5.5, 0.25) {};
		\node  (12) at (9.25, 1.5) {};
  \fill(5)circle(3 pt);
  \fill(6)circle(3 pt);
  \fill(7)circle(3 pt);
  \fill(8)circle(3 pt);
  \fill(9)circle(3 pt);
  \fill(10)circle(3 pt);
  \fill(11)circle(3 pt);
  \fill(12)circle(3 pt);
		\node  (13) at (7, 0.75) {};
		\node  (14) at (10, 1.75) {};
		\node  (15) at (12.25, -0.5) {\(t\)};
		\node  (17) at (0.5, 2.75) {\(q\)};
		\node  (18) at (2, -1.75) {\((t_0,q_0)\)};
		\draw (3.center) to (2.center)[<-];
		\draw (0.center) to (4.center)[->];
		\draw [in=180, out=45, looseness=0.75] (5.center) to (6.center);
		\draw [in=-135, out=45, looseness=0.75] (7.center) to (8.center);
		\draw [in=-135, out=45, looseness=0.75] (11.center) to (13.center);
		\draw [in=-180, out=45] (13.center) to (9.center);
		\draw [in=150, out=60, looseness=1.25] (12.center) to (14.center);
		\draw [in=-135, out=-30] (14.center) to (10.center);
\end{tikzpicture}
}
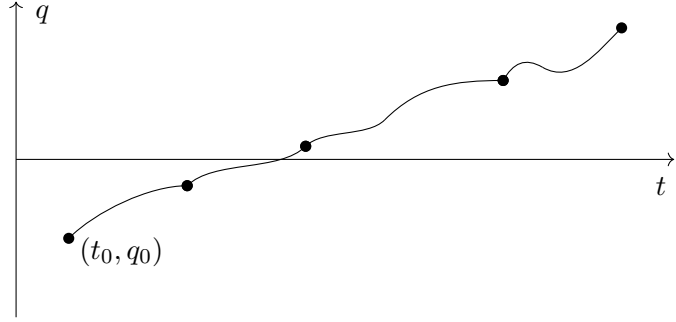
\centerline{\tiny \bf  Fig. 2 -- The graph of a solution in $\bR \times \cQ$ corresponding to  a stepped graph}
\item $\h$ is  determined by its odd regular arcs (the even arcs  are  just segments and  are completely determined  by  the endpoints of the odd arcs). 
\end{itemize}
Conversely, if  $u: [0, T] \to K \subset \bR^m $ is a  piecewise constant control curve   in  $ \cK \subset \bR^m$ and takes  constant values   $c_0, c_1, \ldots, c_r \in \cK$ on the  
subintervals $$[0, t_1)\ , \ (t_1, t_2)\ , \ldots\ , \ (t_r, T] \qquad \text{with} \qquad [0, T] = \bigcup_{j = 0}^r [t_j, t_{j+1}]$$ and denoting by $q(t)$, $t \in [0, T]$,  the  solution to \eqref{Thesystem-1} corresponding to $u(t)$ and the initial value $q(0) = \wt q_0$,  there is just one stepped  graph $\h = \h_1 \ast \h_2 \ast  \ldots\ast \h_{2 r + 1}$ in $ \cM$,  whose odd regular arcs  $\h_{2 k + 1}$ are  determined   by the restrictions  $q|_{[(t_k, t_{k+1}]}$, $0 \leq k \leq r$. The  projection on $\bR \times \cQ$ of $\h$  is   the graph $\{(t, q(t))\ ,\ t \in [0, T]\} \subset \bR \times \cQ$ of such a solution.  \par
\medskip 
Given a point $\overline x = (\overline t, \overline q, \overline w) \in \cM$,  the  {\it $\cM$-attainable set   of $\overline x$ in the  time $T$} is the  set  
 \begin{multline} \Attain^{(T)}_{\overline x} {=} \bigg\{ y \in \cM\,:\, y\ \text{is the final point of  a stepped  graph}\ \h \\
 \text{that  starts from  $\overline x$ and is  such that}\   \pi^\bR(\h) =  [0, T] \bigg\}\ .
 \end{multline}
 From the previous discussion and \eqref{3_6}, it follows  that 
 \beq\label{the53}  \cC_{\overline q }^{(A,B)p.c.}   =  \cO_{\overline q}^{(\bA,\bB)p.c.} =\bigcup_{T> 0, \overline w \in \cK} \pi^{\cQ}\left( \Attain^{(T)}_{\overline x \= (0, \overline q, \overline w)}\right)\ .\eeq
\begin{rem}\label{remark}
   Due to the above relation,  whenever there exists a point  $\overline w \in \cK$ and a time $T> 0$, such that  $ \Attain^{(T)}_{\overline x \= (0, 0, \overline w)}$  projects onto a neighbourhood of $0_{\bR^n}$  via $\pi^{\cQ}$,  it follows immediately  that  $   \cC_{\overline q }^{(A,B)p.c.}   = \cO_{\overline q}^{(\bA,\bB)p.c.}$ contains  a neighbourhood of the origin. 
   \end{rem}
 Our second proof of the sufficiency part of the  Kalman Theorem consists precisely in exploiting  Remark \ref{remark}, namely in showing that the Kalman condition   implies that, for    $\overline w  = 0$,  the projection  of $ \Attain^{(T)}_{\overline x \= (0, 0, 0)}$ onto  $\cQ$  has the desired property for any choice of $T> 0$.  \par
 \medskip 
 \subsection{Final points of stepped graphs as points in   orbits of  flows of vector fields} \label{section6.2} Consider  a stepped graph $\h = \h_1 \ast \h_2 \ast  \ldots\ast \h_{2 r + 1}$ in $\cM = \bR \times \cQ \times \cK$, corresponding to a solution to \eqref{Thesystem-1}  for a piecewise constant control $u: [0, T] \to \cK$ with values $c_k \in \cK$, $0 \leq k \leq r$, on  time intervals  $[t_k, t_k + \s_k]$, with $t_0 = 0$ and $t_k \= t_{k -1} + \s_{k -1}$. Let also $\overline x$ and $\overline y$ be the initial and final points of $\h$. \par
 \smallskip
 We  observe   that   each   odd arc  $\h_{2 k +1}$, $0 \leq k \leq r$, 
 is 
 an integral curve of the vector field 
$$\bT = \frac{\p}{\p t} + \left(\bA^i_j q^j  + \bB^i_a u^a\right) \frac{\p}{\p q^i} =   \frac{\p}{\p t} - \left(A^i_j q^j  + B^i_a u^a\right) \frac{\p}{\p q^i}  \ ,$$
while each even  arc $\h_{2 (k +1)}$ (that is,  each of the  segments  that join the points $(t_{k +1}, \wh q_k, c_k)$ to   $(t_{k+1}, \wt q_{k+1} = \wh q_k, c_{k+1})$) are   integral curves,   with integration parameter running in $ [0,1]$, 
of  the vector fields  
$$\l_{(k)}^a  \frac{\p}{\p u^a}\ , \qquad \l_{(k)}^a \= c_{k}^a - c_{k-1}^a\ , \qquad 1 \leq k \leq r\ . $$
Setting   $W_a \= \frac{\p}{\p u^a}$,  we conclude that  the final point $\overline y$  is  equal to 
\beq  \label{compos} \overline y = \Phi^\bT_{\s_r} \circ \Phi_{1}^{\l_{(r)}^a W_a} \circ \Phi^\bT_{\s_{r}} \circ  \ldots\circ   \Phi_{1}^{\l_{(2)}^a W_a} \circ \Phi^\bT_{\s_1} \circ \Phi_{1}^{\l_{(1)}^a W_a}  \circ \Phi^\bT_{\s_0}(\overline x)\ .\eeq
The next lemma  shows that such   expression  is equivalent  to another, much  more convenient for our purposes.
In its statement, for any  $s\in \bR$ and  any  vector field $X $ on $\cM = \bR \times \cQ \times \cK$,   we adopt the notation $X^{[s]}$ to indicate the  vector field  defined by 
$$X^{[s]} = \Phi^\bT_{s*}(X)\ , $$
that is the push-forward of $X$ by the diffeomorphism $ \Phi^\bT_s$.
 We recall that such a pushed-forward vector field   $X^{[s]}$ is  the  unique vector field whose  integral curves  are     the curves  in $\cM$ defined by 
 $$\g(t) = \Phi^\bT_s \circ \Phi^X_t \circ \Phi^\bT_{-s}(y)\ , \qquad y \in \cM\qquad (\  \text{recall the relation \eqref{useful}} \ ) \ .$$
\begin{lem} \label{lemma51}  Setting
 $\t_k  \=   \sum_{j = k}^r \s_j$,   $1 \leq k \leq r$,    the   point  $\overline y$ in \eqref{compos}  is also  equal to  the point
\beq  \label{compos1} \overline y = \big( \Phi_{1}^{\l_{(r)}^a W^{[ \t_r]}_a} \circ  \Phi_{1}^{\l_{(r-1)}^a W^{[  \t_{r -1}]}_a} \circ 
 \ldots \ \circ \Phi_{1}^{\l_{(1)}^a W^{[	 \t_1]}_a} \big) (\overline x') \ \  \text{where}\  \overline x' \= \Phi^\bT_{T}(\overline x)\ .
\eeq
In particular,  if   $\overline x = (0,\overline q =  0, \overline w = 0)$,
\beq  \label{compos2} \overline y = \big( \Phi_{1}^{\l_{(r)}^a W^{[ \t_r]}_a} \circ  \Phi_{1}^{\l_{(r-1)}^a W^{[ \t_{r -1}]}_a} \circ 
 \ldots \ \circ \Phi_{1}^{\l_{(1)}^a W^{[\t_1]}_a} \big) (\overline x)\ .
\eeq
\end{lem}
 \centerline{
\begin{tikzpicture}
\draw[<->, line width = 0.7] (4,5) to (4,2.5) to (12.5,2.5);
\draw[->, line width = 0.7]  (4,2.5) to  (2,0.5);
\node at  (12, 2.3) { \tiny$t$};
\node at  (3.8, 4.8) {\tiny $w$};
\node at  (2.3, 0.5) {\tiny $q$};
\begin{scope}[very thick,decoration={
    markings,
    mark={at position 0.3 with {\arrow{latex}}}} ] 
\draw [line width = 1, purple, densely dotted, postaction={decorate} ] (3.2,0.8)  to  [out=-5, in=200] (4, 1.4)   to [out=20, in=200] (9.7, 1.6) ; 
\end{scope}
\begin{scope}[very thick,decoration={
    markings,
    mark={at position 0.5 with {\arrow{latex}}}} ] 
\draw [line width = 1, purple, densely dotted, postaction={decorate} ] (3.2,0.8)  to  [out=-5, in=200] (4, 1.4)   to [out=20, in=200] (9.7, 1.6) ; 
\end{scope}
\begin{scope}[very thick,decoration={
    markings,
    mark={at position 0.7 with {\arrow{latex}}}} ] 
\draw [line width = 1, purple, densely dotted, postaction={decorate} ] (3.2,0.8)  to  [out=-5, in=200] (4, 1.4)   to [out=20, in=200] (9.7, 1.6) ; 
\end{scope}
\node at  (7.5, 1.1) {\color{purple} \tiny  \bf   Curve $s \mapsto \Phi^\bT_s(\overline x)$, $s \in [0, \s_1 + \ldots + \s_r] =  [0, T]$
};
\draw [line width = 1, black]  (9.7, 1.6)   [out=80, in=220] to  (9.9, 2.4)  [out=100, in=280] to     [out=90, in=260]  (10.5, 3.4)   ; 
\draw[fill, black]  (9.7, 1.6) circle [radius = 0.05];
\draw[fill, black]  (9.9, 2.4) circle [radius = 0.05];
\draw[fill, black]  (10.5, 3.4) circle [radius = 0.05];
\node at  (9.9, 1.5) {\color{black} \tiny $\overline x'$};
\node at  (10.7, 3.4) {\color{black} \tiny $\overline y$ };
\begin{scope}[very thick,decoration={
    markings,
    mark={at position 0.6 with {\arrow{latex}}}} ] 
\draw [line width = 0.7, purple, postaction={decorate} ](3.2,0.8)  to  [out=-5, in=200] (4, 1.4)   ; 
\draw [line width = 0.7, red, postaction={decorate}](4,1.4)  to  [out=90, in=270]  (4, 2.35)   ;
\draw [line width = 0.7, purple, postaction={decorate} ](4,2.35)  to   [out=10, in=190] (6, 2.45)   ; 
\draw [line width = 0.7, red, postaction={decorate}](6,2.45)  to   [out=90, in=270] (6, 3.45)   ;
\draw [line width = 0.7, purple, postaction={decorate}](6,3.45)  to   [out=10, in=190] (10.5, 3.4)   ; 
\end{scope}
\begin{scope}[very thick,decoration={
    markings,
    mark={at position 0.3 with {\arrow{latex}}}} ] 
\draw [line width = 0.7, purple, postaction={decorate}](6,3.45)  to   [out=10, in=190] (10.5, 3.4)   ; 
\end{scope}
\draw[fill, purple]  (4, 1.4) circle [radius = 0.05];
\draw[fill, purple]  (4, 2.35) circle [radius = 0.05];
\draw[fill, purple]  (6, 2.45) circle [radius = 0.05];
\draw[fill, purple]  (6, 3.45) circle [radius = 0.05];
\draw[fill, purple]  (3.2, 0.8) circle [radius = 0.05];
\node at  (3.35, 0.66) {\color{purple} \tiny $\overline x$};
\node at (1, 3.5) {\color{red} \tiny \it  Curves $s \mapsto \Phi_{s}^{\l_{(j)}^a W_a} $};
\draw[->, line width = 0.3, red] (1.1, 3) to [out = 0, in= 180] (4.1, 2);
\draw[->, line width = 0.3, red] (1.1, 3) to [out = 0, in= 170] (5.8, 3);
\node at (11, 0) {\color{black} \tiny \it Curves $s \mapsto \Phi_{s}^{\l_{(j)}^a W^{[\t_j]}_a} $};
\draw[->, line width = 0.3, black] (11, 0.1) to [out = 90, in= 0] (10, 2);
\draw[->, line width = 0.3, black] (11, 0.1) to [out = 90, in= 0] (10.5, 3);
 \end{tikzpicture}
 }
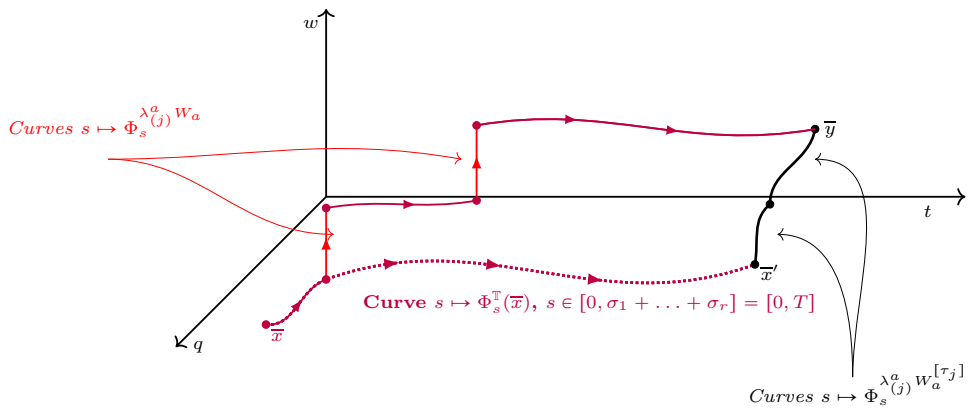
 \centerline{\tiny \bf Fig. 3 --  The black curve from $\overline x'$ to $\overline y$ is obtained  changing the  }
  \centerline{\tiny \bf   composition of flows  \eqref{compos}  into a new one by an iterated  application  of   \eqref{iteratively}.}
\begin{pf}   By the identity \eqref{useful},  any   composition  of  flows of  the   form $\Phi^\bT_{s} \circ \Phi_{t}^{X}$ can be replaced by   
\beq \label{iteratively} \Phi^\bT_{s}  \circ \Phi^X_t  = \Phi^{\Phi^\bT_{s*} (X)}_t \circ \Phi^\bT_{s}  = \Phi^{X^{[s]}}_t \circ \Phi^\bT_{s}\ .\eeq
Applying this identity iteratively,  after a finite number of steps  formula  \eqref{compos} takes the form     (see Fig.3) 
\begin{multline}   \overline y = \left( \Phi_{1}^{\l_{(r)}^a W^{[\s_r]}_a} \circ  \Phi_{1}^{\l_{(r-1)}^a W^{[ \s_{r -1} + \s_r]}_a} \circ 
 \ldots \ \circ \Phi_{1}^{\l_{(1)}^a W^{[\s_r + \s_{r-1} + \ldots \s_1]}_a} \right) \circ
  \\
 \circ \Phi^\bT_{\s_r+ \ldots + \s_1+ \s_0}  (\overline x) \ .
 \end{multline} 
Since $\s_r+ \ldots + \s_1+ \s_0 = T$,   \eqref{compos1} follows. The last claim is an immediate consequence of \eqref{compos1} and the fact that $\Phi^\bT_t(0) = 0$ for all $t \in \bR$.
\end{pf}
From Lemma \ref{lemma51} if follows that if $\overline x_o = (0,0, 0)$,  a point  $\overline y $ belongs to $\Attain^{(T)}_{\overline x_o}$ for some $T> 0$ if and only if it has the form \eqref{compos2}
for a finite set of   constants $\l_{(k)}  =  c_{k} - c_{k-1}$, $1 \leq k \leq r$ (which  correspond to  constant controls   $c_k \in \cK  \subset  \bR^m$) and a finite sequence of   real  numbers  $\t_k$  satisfying the  inequalities $T >  \t_1 >  \t_2 > \ldots >  \t_r > 0 $. \par
\medskip
\subsection{Our second proof of the sufficiency part of the Kalman Theorem} \label{section6.3}
Consider the  constant  vector fields on $\cM $  defined recursively by   
\beq \label{518}
\begin{split}
&W^{(0)}_1 \= W_1 =  \frac{\p}{\p u^1} \ ,\quad  W^{(0)}_2 \= W_2 = \frac{\p}{\p u^2}\ ,\qquad  \ldots\ , \qquad  W_m^{(0)} \= W_m = \frac{\p}{\p u^m}\ ,\\
&W_1^{(1)} \= [\bT, W_1] = B_1^i \frac{\p}{\p q^i}\ ,\quad W_2^{(1)} \= [\bT, W_2] = B_2^i \frac{\p}{\p q^i}\ , \ \ldots\ \\
& \hskip 7 cm \ldots \ \ W^{(1)}_m \=   [\bT, W_m] =  B_m^i \frac{\p}{\p q^i}\ ,\\
&W_1^{(2)} \= [\bT, W^{(1)}_1] = A_j^i B_1^j \frac{\p}{\p q^i}\ ,\quad W_2^{(2)} \= [\bT, W^{(1)}_2] = A_j^i B_2^j \frac{\p}{\p q^i}\ , \ \ldots\\
& \hskip 7 cm \ldots \ ,\ W^{(2)}_m \=   [\bT, W^{(1)}_m] =  A^i_j B_m^j \frac{\p}{\p q^i}\ ,\\
& \hskip 6 cm \vdots \\
&W_1^{(\ell)} \= [\bT, W^{(\ell-1)}_1] = A_{j_1}^{i} A_{j_2}^{j_1} \ldots A^{j_{\ell -2}}_{j_{\ell-1}} B_1^{j_{\ell-1}}\frac{\p}{\p q^i}\ , \ \ldots\\
& \hskip 4 cm \ldots \ ,\ W^{(\ell)}_m \=   [\bT, W^{(\ell-1)}_m] =   A_{j_1}^{i} A_{j_2}^{j_1} \ldots A^{j_{\ell -2}}_{j_{\ell-1}} B_m^{j_{\ell -1}} \frac{\p}{\p q^i}\ ,\\
& \hskip 6 cm \vdots
\end{split}
\eeq
Note that the $i-$th component of each vector field $W^{(\ell)}_a$, $1 \leq a \leq m$,  is always equal to  the $i-$th component of the vector $\underset{(\ell-1)-\text{times}}{\underbrace{A{\cdot} \ldots {\cdot} A}}{\cdot}B_a$.\par
\smallskip
Moreover, for any integer  $\ell \geq 0$, let us denote by $n_\ell$ the dimension of the vector space spanned by the $m{\cdot} (\ell+1)$  vector fields $W_a^{(k)}$, $1 \leq a \leq m$, $0 \leq k \leq \ell$
at a point $x_o \in \cM$:
$$n_\ell = \dim_\bR \langle W^{(0)}_{a_0}|_{x_o}, W_{a_1}^{(1)}|_{x_o}, \ldots, W_{a_\ell}^{(\ell)}|_{x_o}\ ,\  1 \leq a_k\leq m\ \rangle$$
(since all  considered  vector fields are constant,   $n_\ell$ is independent of $x_o$). Clearly the sequence  of integers $n_0, n_1, n_2, \ldots$ is non-decreasing  and  each of them is  less than or equal to $m + n = \dim \cK + \dim \cQ$.  We denote by   $\ell_{\text{max}}$   the smallest integer  for which  $n_{\ell_{\text{max}}}$ is maximal.  One can immediately realise that (see \eqref{tre})
\beq \label{importante} n_{\ell_{\max}} = m + \dim \Kalman^{(A, B)}\eeq
and that  {\it $n_{\ell_{\max}} = m + n$ if and only if $ \dim \Kalman^{(A, B)} = n$}.\par
\medskip
Let us now fix an ordered tuple 
$\cB = (E_1, \ldots, E_{n_{\ell_{\max}}})$,
 in which each vector field  $E_A$ is a vector belonging to  the set  $\{W_a^{(\ell)}\}_{\smallmatrix 1 \leq a \leq m\\ 0 \leq \ell \leq \ell_{\max} \endsmallmatrix}$, and such that they constitute  a basis 
for   the following  vector space    at  $x_o \in \cM$ (and thus at any other point):  
\beq  \langle W^{(0)}_{a_0}|_{x_o}, W_{a_1}^{(1)}|_{x_o}, \ldots, W_{a_{\ell_{\max}}}^{(\ell_{\max})}|_{x_o}\ ,\  1 \leq a_j\leq m\ \rangle\ .\eeq
Since  $n_{\ell_{{\max}}}$ is  the maximal dimension, each vector field $W^{(\ell)}_a$, $\ell \geq  \ell_{\max} +1$,  admits a unique expansion with constant coefficients in terms of the vector fields $E_A$, $1 \leq A \leq n_{\ell_{{\max}}}$.  Since each $E_A$ has the form $W_b^{(\ell)}$ for some $1 \leq b \leq m$ and $0  \leq \ell \leq \ell_{\max}$, we have that all   vector fields $W^{(r)}_a$, $r \geq  \ell_{\max} + 1$, admits a unique expansion of the form 
\beq \label{exp1}  W^{(r)}_a = \sum_{\smallmatrix 0 \leq \ell  \leq \ell_{\max}\\ 
1 \leq b \leq m\endsmallmatrix} \lambda^{r|b}_{a|\ell} W^{(\ell)}_b\ ,\qquad r \geq \ell_{{\max}} + 1\ , \eeq
where  each  coefficient $ \lambda^{r|b}_{a|\ell} $ is either  $0$ or equal to the component  of   $W^{(r)}_a$ in the direction of $ W^{(\ell)}_b$ in case  $ W^{(\ell)}_b$  is one of the chosen basis  $\cB = (E_A)$. 
\begin{lem} \label{lemma52} For any $1 \leq a \leq m$ and $s \in \bR$, the vector field $W_a^{[s]} = \Phi^\bT_{s*}(W_a)$  has a unique expansion of the form 
\beq \label{expansion}W_a^{[s]} = \sum_{\ell = 0}^{\ell_{{\max}} }\left((-1)^\ell \frac{s^\ell}{\ell!} \d_a^b +  s^{\ell_{{\max}} + 1 }D_{a|\ell}^b(s) \right)W_b^{(\ell)}\ ,\eeq
where $D_{a|\ell}^b(s) $ is the  sum of the converging series
\beq   
D_{a|\ell}^b(s) \=  (-1)^{\ell_{\max} + 1 } \sum_{\ell = 0}^\infty (-1)^{\ell}  \frac{s^{\ell }}{(\ell + \ell_{{\max}} + 1)!} \lambda_{a|\ell}^{\ell + \ell_{{\max}}+ 1|b}
\eeq 
and $\d_a^b$ is the Kronecker delta. 
In particular,
\begin{itemize}[leftmargin = 18pt]
\item[--] for each  $s\in \bR$, the  vector field $W_a^{[s]}$ is a  finite linear combination of the vector fields $W_b^{(\ell)}$, $1 \leq b \leq m$, $0 \leq \ell \leq \ell_{\max}$,  and
\item[--] for each $s ,s' \in \bR$,  the   diffeomorphism $\Phi^{W^{[s]}_a}_{s'}$ maps the hypersurfaces $\{ t = \text{const.}\}$  into themselves.
\end{itemize} 
\end{lem}
\begin{pf}   We recall that  for any $s \in \bR$ and any vector field $X$ on $\cM$, 
\beq \frac{d}{d s}  \Phi^{\bT}_{s_*}(X)\bigg|_s =  \frac{d}{d h} \Phi^{\bT}_{h_*}(X^{[s]})\bigg|_{h = 0} =  \frac{d}{d h} \Phi^{-\bT}_{h}{}^*(X^{[s]})\bigg|_{h = 0}   =  - [ \bT , X^{[s]}]  \ .\eeq
Due to this, for any   $a \in \{1, \ldots, m\}$, the one-parameter family of vector fields $s \mapsto W_a^{[s]} $,  $s \in \bR$, is the unique solution to the differential problem
\beq \label{equ} \frac{d  W_a^{[s]} }{d s}\bigg|_s = - [\bT, W_a^{[s]}] \ ,\qquad   W_a^{[s = 0]} = W_a\ .\eeq
On the other hand,  if   $W_a^{(\ell)i}$,  $W_a^{(\ell)b}$ denote  the  components of  the vector field $W_a^{(\ell)} = W_a^{(\ell)i} \frac{\p}{\p q^i} + W_a^{(\ell)b} \frac{\p}{\p u^b}$, using the iterative definition   \eqref{518} of such vector field, one can directly check that the power series 
$$ X^i(s)  =  \sum_{\ell = 0}^\infty (-1)^\ell \frac{s^\ell}{\ell!} W_a^{(\ell)i} \ ,\qquad  X^b(s)  \= \sum_{\ell = 0}^\infty (-1)^\ell \frac{s^\ell}{\ell!} W_a^{(\ell)b}$$
 uniformly converges on any  closed   interval  $[s_1, s_2] \subset \bR$. This implies that  the vector field
 $$X(s) = X^i(s) \frac{\p}{\p q^i} + X^b(s) \frac{\p}{\p u^b} $$ is well defined and real analytic for  any $s \in \bR$,  and it is such that 
 $X(0) = W_a$ and  $\frac{d   X(s) }{d s}\big|_s =   - [\bT, X(s)]$.
From  uniqueness  of the solution to \eqref{equ},  we conclude that 
$W^{[s]}_a =  X(s) =  \sum_{\ell = 0}^\infty (-1)^\ell \frac{s^\ell}{\ell!} W_a^{(\ell)}$.  We now recall that each vector field $W_a^{(\ell)}$, $\ell \geq \ell_{{\max}} + 1$, $1 \leq a \leq m$, admits a unique expansion  \eqref{exp1}. Replacing  such expansion in all  terms corresponding to  the powers $s^\ell$, $\ell   \geq \ell_{{\max}} + 1$, we get \eqref{expansion}.  The last two claims are  consequences of  \eqref{expansion} and the fact that each  vector field $W_b^{(\ell)}$ has  trivial component along  $\frac{\p}{\p t}$.
\end{pf}
 For any $\ve >0$ and any  ordered pair of integers $(a, \ell)$, $1 \leq a \leq m$, $0 \leq \ell \leq \ell_{{\max}}$,   let  
 $$\t_{a, \ell} \=   \frac{a + \ell m }{m \cdot (\ell_{{\max}}+1)}\  ,\qquad  \bW_a^{(\ell)} \= \ell! W_{a}^{[\ve \t_{a,\ell}]}\ .$$
Note that  each real number  $\ve \t_{a, \ell}$ belongs to  the interval $ \left[\frac{1}{m (\ell_{{\max}} + 1)}  \ve, \ve \right]$ and the latter is  a subinterval of $(0, T)$ for a prescribed  value of $T$, provided that  $\ve$ is sufficiently small.
  Imposing  the lexicographic order for the  set of  pairs $(a, \ell)$ and denoting each such pair by a single index $ 1 \leq A \leq m \cdot (\ell_{{\max}}+1)$,  we may consider  the
tuples of vector fields 
$$\{ \bW_A\}_{1 \leq A \leq m \ell_{{\max}}} = \{\bW_a^{(\ell)}\}_{\smallmatrix 1 \leq a \leq m\\ 0 \leq \ell \leq \ell_{{\max}}
\endsmallmatrix}\ ,\qquad  \{W_B\}_{1 \leq B \leq m \ell_{{\max}}}  = \{ W_a^{(\ell)}\}_{1 \leq a \leq m, 0 \leq \ell \leq \ell_{{\max}}}$$
 and the squared   matrix  $\cA= \big(\cA_A^B\big)$, with entries    given by  the components of the vector fields $\bW_A$ in terms of the vector fields $W_B$ according to  \eqref{expansion}. Notice that $\cA$ has the  form 
$\cA =\wh  \cA + O(\ve^{\ell_{{\max}}+1})$, where
\begin{multline}  \wh \cA= \left(\begin{array} 
{cccc} 
\wh \cA_1& 0 & \dots & 0\\
0 & \wh \cA_2 & \dots & 0\\
\vdots & \vdots & \ddots & 0\\
0 & 0 & \dots & \wh \cA_m
\end{array}  \right)\ \\
\text{with}\ \wh \cA_a \= \left(\begin{array} {ccccc } 1 & \ve \t_{a,0} & \ve^2  \t^2_{a,0} &  \ldots & \ve^{\ell_{{\max}}} \t_{a, 0}^{\ell_{{\max}}}\\
 1 & \ve \t_{a,1} &  \ve^2 \t^2_{a,1} & \ldots & \ve^{\ell_{{\max}}} \t_{a, 1}^{\ell_{{\max}}}\\
 \vdots & \vdots & \vdots & \ddots & \vdots\\
  1 & \ve  \t_{a,\ell_{{\max}}} & \ve^2 \t^2_{a,\ell_{{\max}}} & \ldots & \ve^{\ell_{{\max}}} \t_{a,\ell_{{\max}}}^{\ell_{{\max}}}\end{array}
  \right)\ .
\end{multline}
 and  $ O(\ve^{\ell_{{\max}}+1})$ stands for a squared matrix whose  entries  are infinitesimal of order $\ell_{{\max}}+1$ with respect to $\ve$.  
\begin{lem} \label{lemma53} For any  $T> 0$, there exists  $\ve \in \left(0, \frac{T}{2}\right)$ such that the  matrix $\cA$ is invertible. 
\end{lem}
\begin{pf}  Consider the matrix
  \beq    \cB = \left(\begin{array} 
{cccc} 
  \cB_1& 0 & \dots & 0\\
0 &   \cB_2  & \dots & 0\\
\vdots & \vdots & \ddots & 0\\
0 & 0 & \dots &   \cB_m
\end{array}  \right)\  ,\ \ 
  \cB_a\= \left(\begin{array} {ccccc } 1 &  \t_{a,0} &   \t^2_{a,0} &  \ldots & \t_{a, 0}^{\ell_{{\max}}}\\
 1 &  \t_{a,1} &   \t^2_{a,1} & \ldots & \t_{a, 1}^{\ell_{{\max}}}\\
 \vdots & \vdots & \vdots & \ddots & \vdots\\
  1 &   \t_{a,\ell_{{\max}}} &  \t^2_{a,\ell_{{\max}}} & \ldots &  \t_{a,\ell_{{\max}}}^{\ell_{{\max}}}\end{array}
  \right)\ .
\eeq
Factoring out an appropriate   power $\ve^r$, $ 0 \leq r \leq \ell_{{\max}}$,  from the terms of  each column of $\cA$, we get that  
$$\det \cA =\left(\prod_{0 \leq r \leq  \ell_{\max}} \ve^r\right)^m  \det\left( \cB + O(\ve)\right)\ ,$$
where $O(\ve)$ denotes  a matrix whose entries are infinitesimal of the same order of $\ve$.
Therefore,  by continuity  and the well-known  formula for the  determinants of Vandermonde matrices (see e.g. \cite{HJ}),  
$$\det \cA(\ve) =  \left(\prod_{0 \leq r \leq  \ell_{\max}} \ve^r\right)^m \prod_{1 \leq a \leq m} \left(\prod_{1\leq \ell <  \ell' \leq \ell_{{\max}} }
(\t_{a, \ell'} - \t_{a, \ell} )+ O(\ve)\right)$$
where now  $O(\ve)$ stands for   a {\it function} which is an infinitesimal of the same order of $\ve$.  Since each  factor  $\prod_{1\leq \ell <  \ell' \leq \ell_{{\max}} }
(\t_{a, \ell'} - \t_{a, \ell} )+ O(\ve)$ tends to a non-zero value for $\ve $ tending to $0$, the lemma follows. 
\end{pf}
\begin{theo} \label{theprop1}  Assume that $\cK \subset \bR^m$ is  an open set  and that  $ \cC_{\overline q}^{(A,B)p.c.} $  contains a point $q$  that   reaches $\overline q$ in a time $T> 0$  through a solution  $q^{(q, u_o)}$   with $u_o(t)$  piecewise constant. If $\Kalman^{(A, B)} = n$,  then $  \cC_{\overline q}^{(A,B)p.c.} $ contains a neighbourhood of the point $ q$.  \par
\end{theo}
\begin{pf}   Consider an initial   subinterval $[0, T']$ of  $[0,T[$, on which the control $u_o(t)$ is constant and let $\overline q' \= q(T')$. Since $\overline q'$ is in $  \cC_{\overline q}^{(A,B)p.c.} $,  each  point  which is controllable to $\overline q'$ is  automatically a point of  $  \cC_{\overline q}^{(A,B)p.c.} $. Hence, being $q $ controllable to $\overline q'$,  if we show that $  \cC_{\overline q'}^{(A,B)p.c.} $ contains a neighbourhood of $q$, then we immediately get that also $  \cC_{\overline q}^{(A,B)p.c.} $ contains a neighbourhood of $q$. This fact shows that there is no loss of generality if we prove the theorem just under the stronger assumption that $u_o(t)$ is  constant. 
\par
Let $q$,  $ \cC_{\overline q}^{(A,B)p.c.}$ and $u_o(t)$ as in the hypothesis and assume that  $u_o(t)$ is constant.  
Pick  $\ve \in (0, T)$  and consider the   map 
\begin{multline} f^{(\ve)}:\bR^{m {\cdot} (\ell_{{\max}} + 1)} \longrightarrow\cM\\
f^{(\ve)}(s_{1, 0} , s_{2,0},  \ldots, s_{a, \ell}, \ldots, s_{m, \ell_{{\max}}})
\=  \bigg(  \Phi_{ s_{1,0}}^{W_1^{[\ve \t_{1,0}])}} \circ   \Phi_{s_{2, 0}}^{W_2^{[\ve \t_{2,0}]}}\circ \ldots 
 \circ  \Phi_{s_{a, \ell}}^{W_a^{[\ve \t_{a,\ell}]}} \circ \ldots\\
 \ldots \circ  \Phi_{s_{m, \ell_{{\max}}}}^{W_{m-1}^{[\ve\t_{m-1, \ell_{{\max}}}]}} \circ  \Phi_{ s_{m, \ell_{{\max}}}}^{W_m^{[\ve \t_{m, \ell_{{\max}}}]}} \circ  \Phi_T^\bT \bigg)(0, \overline q, c)\ .
\end{multline}
By construction and Lemma \ref{lemma52}, the following properties hold: 
\begin{itemize}[leftmargin = 18pt]
\item[(i)]  $f^{(\ve)}(0) =  \Phi_{T }^\bT(0, \overline q, c)  = (T, q, c)$ and for any $s \in \bR^{m {\cdot} (\ell_{{\max}} + 1)}$, the point $f^{(\ve)}(s)$ is in the hypersurface $\{ t = T\} \subset \cM$; 
\item[(ii)]  By  Lemma \ref{lemma51},  each   point  of the form $\overline y  = f^{(\ve)}(s) $ belongs to  $\Attain^{(T)}_{\overline x_o}$, $x_o = (0, \overline q, c)$; 
\item[(iii)] The columns of the Jacobian  $Jf^{(\ve)}\big|_{0}$ are  the components of the vector fields 
\beq \frac{d  f^{(\ve)}(0 , 0,  \ldots, 0, s_{a, \ell} = t, 0, \ldots, 0)}{d t}\bigg|_{t = 0} = \ell! W_{a}^{[\ve \t_{a,\ell}]} =  \bW_a^{(\ell)}\ \ . \eeq
\end{itemize}
By iii) and Lemma \ref{lemma53},  for any sufficiently small $\ve> 0$, the  rank of  $Jf^{(\ve)}\big|_{0}$ is equal to the dimension of the vector space spanned by the vector fields $W_a^{(\ell)}$, i.e. $m + \dim \Kalman^{(A, B)} $.  Such a rank  is equal to $m + n = \dim \cK + \dim \cQ$ if and only if $ \dim \Kalman^{(A, B)} = n$. If this is the case, by the Inverse Function Theorem  the image of  $f^{(\ve)}$ contains a neighbourhood of the point $y_o = (T, q, c)$ in the hypersurface 
 $\{ t = T\} \subset \cM = \bR \times \cQ \times \cK$. Since the map $\pi^\cQ: \cM \to \cQ$ has maximal rank at all points of    $\{ t = T\} $, such a neighbourhood of $y_o$  projects onto an open neighbourhood of $q$ and,  by \eqref{the53},  such an open set  is a neighbourhood of $q$ which is entirely included in $\cO_{\overline q }^{(\bA,\bB)p.c.} = \cC_{\overline q}^{(A,B)p.c.} $.  
\end{pf}

Our second proof of the sufficiency part of the Kalman Theorem, described at the beginning of this section,  is now a trivial consequence of Theorem \ref{theprop1}:  It suffices to observe  that such a theorem  applies to  the case in which $q = 0_{\bR^n}$  and $\bar q = 0_{\bR^n}$, because  the origin can be considered as  joined to itself  by means the trivial solution of \eqref{Thesystem-1}, determined by the constant control
  $u_o(t) \equiv 0$, which  is an admissible control by the hypothesis of Theorem  \ref{Kalman-criterion}. \\
\par
\section*{Declarations}
\noindent \textbf{Ethical Approval}: Not applicable.\\
\textbf{Funding}: Not apllicable.\\
\textbf{Acknowledgments}: F. Bagagiolo is a member of INDAM-Gruppo Nazionale per l'Analisi Matematica la Probabilit\`a e le loro Applicazioni; M. Zoppello is a member of INDAM-Gruppo Nazionale di Fisica Matematica.
F. Bagagiolo thanks Paolo Venturini for discussions on some parts of the subjects.

\end{document}